\documentclass[10pt]{amsart}
\usepackage{graphicx}
\usepackage{amscd}
\usepackage{amsmath}
\usepackage{amsthm}
\usepackage{amsfonts}
\usepackage{amssymb}
\usepackage{mathrsfs}
\usepackage{bm}
\usepackage{enumerate}
\usepackage{amsrefs}
\usepackage{xcolor}
\usepackage[colorlinks, citecolor=blue, linkcolor=red, pdfstartview=FitB]{hyperref}
\usepackage[latin1]{inputenc}

\usepackage{tikz}
\usepackage{tikz-cd}
\usepackage{caption}
\usetikzlibrary{decorations.markings}
\usepackage{lipsum}
\usepackage{mathrsfs}

\usepackage{marginnote}	% for editing
\usepackage{soul} 			% for indicating new material

\newcommand{\bM}{{\mathbb{M}}}

\newcommand{\bR}{{\mathbb{R}}}

\newcommand{\fK}{{\mathfrak{K}}}

\newcommand{\rC}{\mathrm{C}}

\newcommand{\eps}{\varepsilon}
\renewcommand{\phi}{\varphi}

\newcommand{\upchi}{{\raise.35ex\hbox{$\chi$}}}

\newcommand{\mycomment}[1]{}

\newcommand{\id}{\operatorname{id}}

\newcommand{\Bor}{\operatorname{Bor}}

\newtheorem{lemma}{Lemma}[section]
\newtheorem{theorem}[lemma]{Theorem}
\newtheorem{proposition}[lemma]{Proposition}
\newtheorem{corollary}[lemma]{Corollary}
\newtheorem{theoremx}{Theorem}

\theoremstyle{definition}

\date{\today}

\author{Rapha\"el Clou\^atre}
\address{Department of Mathematics, University of Manitoba, Winnipeg, Manitoba, Canada R3T 2N2}
\email{raphael.clouatre@umanitoba.ca\vspace{-2ex}}

\author{Ian Thompson}
\address{Department of Mathematical Sciences, University of Copenhagen, Universitetsparken 5, 2100 Copenhagen, Denmark}
\email{ian@math.ku.dk\vspace{-2ex}}

\thanks{R.C. was partially supported by an NSERC Discovery Grant. I.T. was partially supported by an NSERC CGS-D Scholarship.}

\title[Rigidity of operator systems]{Rigidity of operator systems: tight extensions and noncommutative measurable structures}
\begin{document}
\begin{abstract}
Let $A$ be a unital $\rC^*$-algebra generated by some separable operator system $S$.
More than a decade ago, Arveson conjectured that $S$ is hyperrigid in $A$ if all irreducible representations of $A$ are boundary representations for $S$. Recently, a counterexample to the conjecture was found by Bilich and Dor-On. To circumvent the difficulties hidden in this counterexample, we exploit some of  Pedersen's seminal ideas on noncommutative measurable structures 
and establish an amended version of Arveson's conjecture. More precisely, we show that all irreducible representations of $A$ are boundary representations for $S$ precisely when all representations of $A$ admit a unique \emph{tight} completely positive extension from $S$.  In addition, we prove an equivalence between uniqueness of such tight extensions and rigidity of completely positive approximations for representations of nuclear $\rC^*$-algebras, thereby extending the classical principle of Korovkin--\v Sa\v skin for commutative algebras of continuous functions.
\end{abstract}
\maketitle

%%%%%%%%%%%%%%%%%%%%%%%%%%%%%%%%%%%%%%
%%%%%%%%%%%%%%%%%%%%%%%%%%%%%%%%%%%%%%
\section{Introduction  and preliminaries}\label{S:Introduction}

The noncommutative Choquet boundary, as envisioned by Arveson in his seminal work \cite{arveson1969subalgebras}, is now widely recognized  as a lynchpin in the study of operator systems and  nonselfadjoint operator algebras. Over the last few decades, it has blossomed and proven itself to be the most fruitful mechanism to identify the $\rC^*$-envelope, for instance \cite{dritschel2005boundary},\cite{arveson2008noncommutative},\cite{davidson2015choquet}.

This noncommutative boundary is often the key to unlocking deeper structural information about an operator system, yet it still holds many mysteries. Accordingly, it continues to attract interest and to generate new research avenues \cite{davidson2019noncommutative},\cite{clouatre2020finite},\cite{kennedy2021noncommutative},\cite{davidson2022strongly},\cite{clouatre2023minimal}.
Perhaps the most prominent question in this direction was raised by Arveson himself in 2011 \cite{arveson2011noncommutative}. The \emph{hyperrigidity conjecture}, as it is usually referred to, has been the impetus for a substantial body of work by many hands \cite{kleski2014korovkin},\cite{kennedy2015essential},\cite{dor2018full},\cite{clouatre2018multiplier},\cite{clouatre2018non},\cite{clouatre2018unperforated},\cite{salomon2019hyperrigid},\cite{HK2019},\cite{KR2020},\cite{kim2021},\cite{davidson2021choquet},\cite{CT2021},\cite{clouatre2023boundary},\cite{thompson2024approximate},\cite{PS2024} where partial results and supporting evidence for the conjecture were gathered. 

In spite of this, the conjecture was shown to be false very recently \cite{bilich2024arveson}. Even more strikingly, the counterexample found in  \cite{bilich2024arveson} is not overly exotic or pathological, thereby suggesting that there may be something intrinsically missing in Arveson's original formulation of the conjecture. The main results  that we obtain in this paper offer  insight into what an ``amended" notion of hyperrigidity could be.

Let us be more precise. Throughout, $A$ will denote a unital $\rC^*$-algebra and $S\subset A$ will be an operator system that generates $A$ as a $\rC^*$-algebra. Let $H$ be a Hilbert space and let $\pi:A\to B(H)$ be a unital $*$-representation. We say that $\pi$ has the \emph{Arveson rigidity property} with respect to $S$ if, given a net 
\[
\psi_\lambda:A\to B(H),\quad \lambda\in \Lambda
\]
of unital completely positive maps with the property that 
\[
\lim_{\lambda} \|\psi_\lambda(s)-\pi(s)\|=0,\quad s\in S,
\]
we necessarily have that 
\[
\lim_{\lambda} \|\psi_\lambda(a)-\pi(a)\|=0,\quad a\in A.
\]
In other words, $\pi$ has the Arveson rigidity property whenever a net of unital completely positive maps that approximates $\pi$ pointwise on $S$ (in the norm topology) must necessarily approximate $\pi$ pointwise on the entire $\rC^*$-algebra $A$. 

Upon choosing a constant net, we see that if $\pi$ has the Arveson rigidity property, then it automatically has the so-called \emph{unique extension property} with respect to $S$: the only unital completely positive map $\psi:A\to B(H)$ agreeing with $\pi$ on $S$ is $\pi$ itself. Irreducible $*$-representations with the unique extension property with respect to $S$ are called \emph{boundary representations} for $S$. When $A$ is commutative, these boundary representations are simply the characters of evaluation at points in the Choquet boundary of $S$. For this reason, the collection of boundary representations for $S$ is often thought of as the noncommutative Choquet boundary of $S$, which we alluded to above.  A key observation of Arveson \cite[Theorem 2.1]{arveson2011noncommutative} allows one to connect this boundary to the Arveson rigidity property.

\vspace{3mm}

\noindent \textbf{The Arveson rigidity principle. }
The following statements are equivalent.
\begin{enumerate}[\rm (i)]
\item Every unital $*$-representation of $A$ has the unique extension property with respect to  $S$.
\item Every  unital $*$-representation of $A$ has  the Arveson rigidity property with respect to $S$.
\end{enumerate}

\vspace{3mm}

When either of these equivalent statements hold, we say $S$ is \emph{hyperrigid} in $A$.

Another piece of insight from \cite{arveson2011noncommutative} was to draw a connection with classical approximation theory. To properly frame this, we require a variation on Arveson's rigidity property. Let $B$ be another unital $\rC^*$-algebra and let $\pi:A\to B$ be a unital $*$-representation. We say that $\pi$ has the \emph{Korovkin rigidity property} with respect to $S$ if, given a net
\[
\psi_\lambda:A\to \pi(A),\quad \lambda\in \Lambda
\]
of unital completely positive maps with the property that 
\[
\lim_{\lambda} \psi_\lambda(s)=\pi(s) \quad \text{weakly in } \pi(A) \quad \text{for each } s\in S
\]
we necessarily have that 
\[
\lim_{\lambda} \psi_\lambda(a)=\pi(a) \quad \text{weakly in } \pi(A) \quad \text{for each } a\in A.
\]
If the weak topology of $\pi(A)$ is replaced with the norm topology in the previous definition, then we instead say that $\pi$ has the \emph{strong Korovkin rigidity property} with respect to $S$. We shall require both variants in our work below.

We emphasize that the strong Korovkin rigidity property is a priori weaker than the Arveson rigidity property, as the ranges of the maps $\psi_\lambda$ are restricted to lie in $\pi(A)$.  

When $A=\rC([0,1])$, it was shown by Korovkin \cite{korovkin1953convergence} that the identity representation has the strong Korovkin rigidity property with respect to the operator system of quadratic polynomials, hence justifying our choice of terminology. Later,  \v{S}a\v{s}kin \cite{vsavskin1967mil} shed light on this phenomenon by providing a criterion for the identity representation to display such rigidity with respect to any separable  function system. The separability assumption was removed in \cite[Theorem 5.3]{davidson2021choquet}, thus yielding the following.

\vspace{3mm}

\noindent \textbf{The Korokvin--\v{S}a\v{s}kin rigidity principle. }
Assume that $A$ is commutative. Then, the following statements are equivalent.
\begin{enumerate}[\rm (i)]
\item Every irreducible $*$-representation of $A$ is a boundary representation for $S$.
\item The identity representation of $A$ has the strong Korovkin rigidity property with respect to $S$.
\item Every unital $*$-representation of $A$ has  the strong Korovkin rigidity property with respect to $S$.
\end{enumerate}

\vspace{3mm}

Upon comparing the two aforementioned rigidity principles, a fundamental questions arises.

\vspace{3mm}

\noindent \textbf{Question A.} 
Can the strong Korovkin rigidity property be improved to the Arveson rigidity property in the Korokvin--\v{S}a\v{s}kin rigidity principle?

\vspace{3mm}

 Irrespective of commutativity, the answer was conjectured to be affirmative by Arveson himself in \cite{arveson2011noncommutative}, at least in the separable case:

\vspace{3mm}

\noindent \textbf{Arveson's hyperrigidity conjecture. }
Let $A$ be a unital  $\rC^*$-algebra generated by some separable operator system $S$. Then, the following statements are equivalent.
\begin{enumerate}[\rm (i)]
\item Every irreducible $*$-representation of $A$ is a boundary representation for $S$.
\item The operator system $S$ is hyperrigid in $A$, i.e. every unital $*$-representation for $A$ has the unique extension property with respect to  $S$.
\end{enumerate}

\vspace{3mm}

Despite Arveson's conjecture having been observed and confirmed for many classes of operator systems, it was very recently found that the conjecture fails \cite{bilich2024arveson}, even when $A$ is of type $I$. Whether or not the statement is valid when $A$ is commutative is still unknown; see \cite{davidson2021choquet} and \cite{clouatre2023boundary} for some recent developments in this direction.

In Section \ref{S:Meas}, we show how to amend Arveson's conjecture so as to make it true, based on the following notion. Given a unital $*$-representation $\pi:A\to B(H)$, we say that it has the \emph{unique tight extension property}  with respect to $S$ if  the only unital completely positive map $\psi:A\to \pi(A)''$ agreeing with $\pi$ on $S$ is $\pi$ itself.  (Here and throughout, for a set $X\subset B(H)$ we denote by $X'$ its commutant). This condition is weaker than the aforementioned unique extension property, as we only consider extensions $\psi$ which are ``tight", in the sense that $\psi(A)\subset \pi(A)''$. The usefulness of restricting the range of the extensions to lie in the von Neumann algebra $\pi(A)''$ was also recognized  in \cite{kleski2014korovkin}.

The following is found in Theorem \ref{T:tightHR}, and  is our first main result.

\begin{theoremx}\label{T:A}
Assume that $S$ is separable. Then, the following statements are equivalent.
\begin{enumerate}[{\rm (i)}]
\item Every irreducible $*$-representation of $A$ is a boundary representation for $S$. 
\item Every unital $*$-representation of $A$ has the unique tight extension property with respect to $S$.
\end{enumerate}
\end{theoremx}

This result shows how Arveson's original conjecture can be modified in order to avoid the pitfalls from the counterexample of Bilich and Dor On \cite{bilich2024arveson}.  Moreover, it improves on \cite[Corollary 3.3]{kleski2014korovkin}, by removing the requirement therein that $A$ be of type $I$. Notably, our techniques do not require direct integration machinery, and as a whole appear to be completely distinct from that used in \cite{kleski2014korovkin}. 

Indeed, we utilize instead Pedersen's work on noncommutative measurable and Borel structures \cite{pedersen1974},\cite{pedersen1976},\cite[Sections 4.5 and 4.6]{eilers2018c}. Our key technical observation (Theorem \ref{T:meastrick}) relies on the existence of what we call \emph{measurable splittings} for a $*$-representation. The classical theorem of Maharam from measure theory implies that such splittings can always be found in the commutative setting (Proposition \ref{P:Mah}). A crucial step in our argument is the development of a noncommutative version of this fact; see Theorem \ref{T:meassplitexist}.

We now turn to another natural question arising from the previous discussion.

\vspace{3mm}

\noindent  \textbf{Question B.} 
Is the Korokvin--\v{S}a\v{s}kin rigidity principle valid in the noncommutative setting?

\vspace{3mm}

There have been related investigations carried out in the past
\cite{choda1963theorems},\cite{priestley1976noncommutative},\cite{robertson1977korovkin},\cite{uchiyama1999korovkin}, but these have no bearing on our precise question. Our second main result is Corollary \ref{C:nuclearrigprinc}, which yields an affirmative answer in the nuclear setting. More precisely, we prove:

\begin{theoremx}\label{T:B}
Assume that $A$ is nuclear and that $S$ is separable. Then, the following statements are equivalent.
\begin{enumerate}[{\rm (i)}]
\item Every irreducible $*$-representation of $A$ is a boundary representation for $S$.
\item The identity representation of $A$ has the Korovkin rigidity property  with respect to $S$.
%\item Every unital $*$-representation of $A$ has the unique tame extension property with respect to $S$.
\item Every unital $*$-representation of $A$ has the Korovkin rigidity property with respect to $S$.
\end{enumerate}
\end{theoremx}

Comparing this result with the classical Korokvin--\v{S}a\v{s}kin rigidity principle above, it follows that the 
Korovkin rigidity property for the identity representation is always equivalent to its strong version, at least for commutative $\rC^*$-algebras.  In Section \ref{S:NormK}, we examine the relationship between the two rigidity properties in detail.  We show in Corollary \ref{C:KShomo} that they are in fact equivalent more generally for all homogeneous $\rC^*$-algebras. Along the way, we establish an invariance principle for the strong Korovkin rigidity property under changes of representations (Theorem \ref{T:lift}), relying on a lifting trick of Arveson \cite{arveson1977notes}.  At the time of this writing, we do not know if the two rigidity properties are always equivalent for the identity representation of nuclear $\rC^*$-algebras.

Finally, we close the paper by revisiting, through the lens of our main results, the recent counterexample to the hyperrigidity conjecture  by Bilich and Dor-On. We prove the following (see Theorem \ref{T:BD}).

\begin{theoremx}\label{T:C}
Let $S$ be the Bilich--Dor-On operator system and let $\pi$ be the Bilich--Dor-On representation. Then, the following statements hold.
\begin{enumerate}[\rm (i)]
\item Every unital $*$-representation of $\rC^*(S)$ has the unique tight extension property with respect to $S$.
\item Every unital $*$-representation of $\rC^*(S)$ has the  Korovkin rigidity property with respect to $S$.
\item The representation $\pi$ has the strong Korovkin rigidity property.
\end{enumerate}
\end{theoremx}

%%%%%%%%%%%%%%%%%%%%%%%%%%%%%%%%%%%%%%
%%%%%%%%%%%%%%%%%%%%%%%%%%%%%%%%%%%%%%
\section{Rigidity and noncommutative measurability}\label{S:Meas}

As discussed in the introduction, Arveson's hyperrigidity conjecture is now known to be false, thanks to \cite{bilich2024arveson}. The aim of this section is to show that there is a natural modification of the original conjecture that makes it true. The crucial idea behind the developments below is to consider noncommutative measurable structures.

\subsection{Uniqueness of tight extensions}\label{SS:msue}
Let $A$ be a unital $\rC^*$-algebra and let $\pi:A\to B(H)$ be a unital $*$-representation. As is well-known, $\pi(A)^{**}$ is a von Neumann algebra.  Throughout the paper,  given another von Neumann algebra $M$ and a bounded linear map $\phi:\pi(A)\to M$, we will denote by $\widehat\phi:\pi(A)^{**}\rightarrow M$ the unique weak-$*$ continuous extension of $\varphi$.

Let  $t\in \pi(A)^{**}$ be a self-adjoint element. Following \cite[Paragraph 4.3.11]{eilers2018c}, we say that $t$ is \emph{universally $\pi$-measurable} if for every $\eps>0$ and every state $\phi$ on $\pi(A)$, there are self-adjoint elements $s_1,s_2\in \pi(A)^{**}$ such that $-s_1\leq t\leq s_2$, $\widehat\phi(s_1+s_2)<\eps$, and both $s_1$ and $s_2$ can be written as  increasing limits in the weak-$*$ topology  of $\pi(A)^{**}$ of self-adjoint nets from $\pi(A)$.

We also require the following notion.  
Let $t\in \pi(A)''$ be a self-adjoint  element.
We say that $\pi$ admits a \emph{measurable splitting at $t$} if there is a positive linear map
$\lambda:\rC^*(\pi(A),t)\to \pi(A)^{**}$ with the following properties.
\begin{enumerate}
\item[{\rm (s1)}] The element  $\lambda(t)$ is universally $\pi$-measurable.
\item[{\rm (s2)}] If $j:\pi(A)\to\pi(A)''$ denotes the inclusion and $\widehat j:\pi(A)^{**}\to \pi(A)''$ is its canonical extension, then $\widehat j\circ \lambda(t)=t$. 
\item[{\rm (s3)}] If $\omega$ is a pure state on $\pi(A)$ such that $\widehat\omega\circ \lambda\neq 0$, then $\widehat\omega\circ \lambda|_{\pi(A)}$ is a pure state as well.
\end{enumerate}

The fundamental lever that will allow us to successfully tweak Arveson's conjecture is the following result.

\begin{theorem}\label{T:meastrick}
Let $A$ be a unital $\rC^*$-algebra and let $S\subset A$ be an operator system such that $A=\rC^*(S)$. Let $\pi:A\to B(H)$ be a unital $*$-representation  and let $\psi:A\to B(H)$ be a unital completely positive map agreeing with $\pi$ on $S$. Let $a\in A$ be a self-adjoint element such that $\psi(a)$ lies in $\pi(A)''$ and $\pi$ admits a measurable splitting  at $\psi(a)-\pi(a)$. If  every irreducible $*$-representation of $A$ is a boundary representation for $S$, then  $\psi(a)=\pi(a)$.
\end{theorem}
\begin{proof}

Put $t=\psi(a)-\pi(a)$ and assume that $t\neq 0$. Upon replacing $t$ by $-t$ if necessary, we may find a weak-$*$ continuous state $\tau$ on $\pi(A)''$ such that $\tau(t)<0$. 

Let $\lambda:\rC^*(\pi(A),t)\to \pi(A)^{**}$ be a measurable splitting for $\pi$ at $t$.
Let $\sigma=\tau|_{\pi(A)}$. Then, we have $\tau\circ \widehat j=\widehat\sigma$ so that $\widehat\sigma(\lambda(t))=\tau(t)<0$ by property (s2). On the other hand, property (s1) yields that $\lambda(t)$ is universally $\pi$-measurable.  Therefore, there is  an increasing net $(\pi(a_i))$ of self-adjoint elements in $\pi(A)$ that converges in the weak-$*$ topology of $\pi(A)^{**}$  to some element $s$ with $\lambda(t)\leq s$ and $\widehat\sigma(s)<0$.

Let $K$ denote the state space of the unital $\rC^*$-algebra $\pi(A)$, equipped with the weak-$*$ topology. Define a function $F:K\to \bR$ by
\[
F(\phi)=\widehat\phi(s),\quad \phi\in K.
\]
Note that $F$ is affine and lower semi-continuous, since it is the increasing pointwise limit of the evaluation maps at $\pi(a_i)$. Hence, $F$ attains its minimum value $\mu\in \bR$, and the set
\[L=\{\phi\in K:F(\phi)=\mu\}\] is a non-empty, weak-$*$ closed face of $K$.
By the Krein--Milman theorem, there is a pure state $\omega$ on $\pi(A)$ with the property that $F(\omega)=\mu$. 
Recall now that $\lambda(t)\leq s$, so
\begin{equation}\label{Eq:omegat}
(\widehat\omega \circ \lambda)(t)\leq F(\omega)\leq  F(\sigma)=\widehat\sigma(s)<0.
\end{equation}
By property (s3), we see that $\widehat\omega \circ \lambda$ restricts to a pure state on $\pi(A)$. 

Finally, note that $\rC^*(\pi(A),t)\subset \rC^*(\psi(A))$; see \cite[Equation (3.1)]{clouatre2018unperforated}. Let $\theta$ be any state on $\rC^*(\psi(A))$ extending $\widehat\omega\circ \lambda$. Then, $\theta(\psi(a)-\pi(a))=(\widehat\omega \circ \lambda)(t)<0$ by \eqref{Eq:omegat} and $\theta|_{\pi(A)}$ is pure.
Invoking \cite[Lemma 3.1]{clouatre2018unperforated}, we may find an irreducible $*$-representation of $A$ that does not have the unique extension property with respect to $S$.
\end{proof}

We are particularly interested in applying the previous result to representations $\pi$ admitting a measurable splitting at $\psi(a)-\pi(a)$ for every self-adjoint $a\in A$. Indeed, in this case Theorem \ref{T:meastrick} implies that the extension $\psi$ agrees with $\pi$ everywhere.  Therefore, our next task is to find sufficient conditions for a representation to enjoy this property. This will be accomplished by adapting a classical tool from measure theory, which we state below to better frame our results.

\begin{proposition}\label{P:Mah}
Let $A$ be a separable, commutative, unital $\rC^*$-algebra and let $\pi:A\to B(H)$ be a cyclic unital $*$-representation. %Let $j:\pi(A)\to \pi(A)''$ denote the inclusion. 
Then, $\pi$ admits a measurable splitting at every self-adjoint $t\in \pi(A)''$.
\end{proposition}
\begin{proof}
There is a compact Hausdorff space $X$ for which $\pi(A)\cong \rC(X)$. Since $\pi$ is cyclic, up to unitary equivalence, we may assume that $H=L^2(X,\mu)$ for some regular Borel probability measure $\mu$, and that $\pi(A)$ acts via  multiplication operators. In this case $\pi(A)''$ is identified with  $L^\infty(X,\mu)$, once again via multiplication operators.

Consider the inclusion $j: \pi(A)\to \pi(A)''$ and its canonical weak-$*$ continuous extension $\widehat j:\pi(A)^{**}\to \pi(A)''$. Let $\Bor(X)\subset \pi(A)^{**}$ be the $\rC^*$-subalgebra of bounded Borel measurable functions on $X$. As explained in \cite[Paragraph 4.3.16]{eilers2018c}, every self-adjoint element in $\Bor(X)$ is universally $\pi$-measurable. Now, the restriction of $\widehat j$ to $\Bor(X)$ is a surjective $*$-homomorphism onto $\pi(A)''$ that maps a function in $\rC(X)$ to its $[\mu]$-equivalence class in $L^\infty(X,\mu)$. 

Finally, fix a self-adjoint element $t\in \pi(A)''$. It is a standard argument (see for instance \cite[Corollary 2, page 91]{arveson2012invitation}) that there is a $*$-homomorphism $\lambda:\rC^*(\pi(A),t)\to \Bor(X)$ such that $\widehat j \circ \lambda=\id$.  We claim that $\lambda$ is a measurable splitting for $\pi$ at $t$. Indeed, it is clear that properties (s1) and (s2) are satisfied by construction of $\lambda$.
Let $\omega$ be a pure state on $\pi(A)$ such that $\widehat\omega\circ \lambda\neq 0$. In particular,  $\widehat\omega(\lambda(1))\neq 0$. Since $\pi(A)$ is commutative, this means that $\omega$ is a character, and thus so is $\widehat\omega\circ \lambda|_{\pi(A)}$. In other words, $\widehat\omega\circ \lambda|_{\pi(A)}$ is a pure state and property (s3) holds.
\end{proof}

Motivated by the argument above, we are thus looking to establish a suitable noncommutative analogue of Maharam's theorem. Classically, the requirement for this result to hold for $L^\infty(X,\mu)$ is that the probability space $(X,\mu)$ be a so-called ``standard Borel space". It stands to reason, then, that Pedersen's work on noncommutative Borel structures \cite{pedersen1974},\cite{pedersen1976} should be useful in our endeavour. Let us be more precise. 

Let $A$ be a $\rC^*$-algebra. Its \emph{enveloping Borel algebra} is the $\rC^*$-subalgebra $\Bor(A)\subset A^{**}$ whose self-adjoint part coincides with the  monotone weak-$*$ sequential closure of the self-adjoint part of $A$; see \cite[Section 4.5]{eilers2018c} for details. An important fact that we require below is that the self-adjoint elements in $\Bor(A)$ are universally $\pi$-measurable, where $\pi$ the identity representation of $A$   \cite[Lemma 4.5.12]{eilers2018c}. We now arrive at our noncommutative extension of Proposition \ref{P:Mah}.

\begin{theorem}\label{T:meassplitexist}
Let $A$ be a separable unital $\rC^*$-algebra, let $H$ be a separable Hilbert space  and let $\pi:A\to B(H)$ be a unital $*$-representation. Then, $\pi$ admits a measurable splitting at every self-adjoint $t\in \pi(A)''$.
\end{theorem}
\begin{proof}
Throughout, we let $j:\pi(A)\to \pi(A)''$ denote the inclusion, with its canonical weak-$*$ continuous extension $\widehat j :\pi(A)^{**}\to \pi(A)''$.

Fix $t\in \pi(A)''$ self-adjoint. It follows from \cite[Corollary 4.5.10]{eilers2018c} that there is $b\in \Bor(\pi(A))$ such that $\widehat{j}(b)=t$. Let $B=\rC^*(\pi(A),b)\subset \Bor(\pi(A))$ and $D=\rC^*(\pi(A),t)\subset \pi(A)''$. Then, $\widehat j|_{B}:B\to D$ is surjective. Since $B$ is a separable $\rC^*$-algebra, there exists an increasing sequence $(f_n)$ of positive contractions in $\ker\widehat j\cap B$ forming a quasicentral approximate unit relative to $B$. Then, $(f_n)$ increases to some projection $e\in \pi(A)^{**}$ in the weak-$*$ topology. Since each $f_n$ lies in $B\subset \Bor(\pi(A))$, it follows that $e\in \Bor(\pi(A))$ as well. In addition, we see that $e$ is central in $B$. Since $B$ contains $\pi(A)$,  $e$ commutes with $\pi(A)$ and hence is central in $\pi(A)^{**}$.

Next, using that $\widehat j(e)=0$, it is readily seen that the restriction $\widehat j:B(1-e)\to D$ is a $*$-isomorphism, and we denote its inverse by $\lambda$, so that 
\begin{equation}\label{Eq:lambda}
\lambda(\pi(a))=\pi(a)(1-e)
\end{equation} 
for every $a\in A$. In particular, we have $\widehat j \circ \lambda=\id$ on $D$. 
We claim that $\lambda$ is the desired splitting.

First, we have $\lambda(t)\in B(1-e)\subset \Bor(\pi(A))$ since $e\in \Bor(\pi(A))$. Hence, property (s1) holds  \cite[Lemma 4.5.12]{eilers2018c}. Furthermore, it is trivial that $\widehat j\circ \lambda(t)=t$ by construction, so that property (s2) is satisfied as well.

Finally, let $\omega$ be a pure state on $\pi(A)$ such that $\widehat\omega \circ \lambda\neq 0$. Let $x\in D$ with $\|x\|=1$ such that $(\widehat\omega\circ \lambda)(x)\neq 0$. By construction, we see that $\lambda(x^*x)\leq 1-e$ so the Schwarz inequality implies 
\[
|(\widehat\omega \circ \lambda)(x)|^2\leq \widehat\omega(\lambda(x^*x))\leq \widehat\omega(1-e).
\]
We infer that $\widehat\omega(1-e)>0$. Since $e$ is central and $\omega$ is pure, this forces $\widehat\omega(1-e)=1$, so that $1-e$ lies in the multiplicative domain of $\widehat\omega$ and $\widehat\omega\circ\lambda|_{\pi(A)}=\omega$ by \eqref{Eq:lambda}. Hence, $\widehat\omega\circ\lambda|_{\pi(A)}$ is pure, and property (s3) holds.
\end{proof}

We now recall some terminology from the introduction. 
Let $A$ be a unital $\rC^*$-algebra and let $S\subset A$ be an operator system such that $A=\rC^*(S)$. 
Let $\pi:A\to B(H)$ be a unital $*$-representation. A unital completely positive map $\psi:A\to B(H)$ agreeing with $\pi$ on $S$  will be said to be a \emph{tight extension} of $\pi$ if $\psi(A)\subset \pi(A)''$. If the only such tight extension is $\pi$ itself, then we say that $\pi$ has the \emph{unique tight extension property} with respect to $S$. We can now state one of the main results of the paper, saying that Arveson's conjecture is valid (in the separable setting) provided one focuses on tight extensions.

\begin{theorem}\label{T:tightHR}
Let $A$ be a  unital $\rC^*$-algebra and let $S\subset A$ be an operator system such that $A=\rC^*(S)$. Then, the following statements are equivalent.
\begin{enumerate}[{\rm (i)}]
\item Every irreducible $*$-representation of $A$ is a boundary representation for $S$. 
\item Every unital $*$-representation of $A$ has the unique tight extension property with respect to $S$.
\end{enumerate}
Then, ${\rm (ii)}\Rightarrow {\rm (i)}$. When $A$ is separable, we have ${\rm (i)}\Leftrightarrow {\rm (ii)}$.
\end{theorem}

\begin{proof}
(ii) $\Rightarrow$ (i): Let $\pi:A\to B(H)$ be an irreducible $*$-representation and let $\psi:A\to B(H)$ be a completely positive extension of $\pi|_S$. Since $\pi(A)''=B(H)$, we see that $\psi$ is a tight extension of $\pi$, so that $\psi=\pi$ by assumption. We conclude that $\pi$ is a boundary representation for $S$.

(i) $\Rightarrow$ (ii):  Assume here that $A$ is separable. Let $\pi:A\to B(H)$ be a unital $*$-representation and let $\psi:A\to \pi(A)''$ be a unital completely positive extension of $\pi|_S$. We can find a collection $\{H_i\}$ of pairwise orthogonal closed subspaces of $H$ such that $H=\bigoplus_i H_i$, along with cyclic $*$-representations $\pi_i:A\to B(H_i)$ such that $\pi=\bigoplus_{i} \pi_i$. Note then that each $H_i$ is separable.  It is readily verified that $\pi(A)''\subset \prod_i \pi_i(A)''$. Consequently, there are unital completely positive maps $\psi_i:A\to \pi_i(A)''$ such that $\psi=\bigoplus_i \psi_i$, and $\psi_i$ and $\pi_i$ agree on $S$ for each $i$.

Since it clearly suffices to show that $\pi_i$ and $\psi_i$ agree on $A$ for each $i$, we may as well assume that $H$ is separable to begin with. In particular, $\pi$ admits a measurable splitting at every self-adjoint element in $\pi(A)''$ by virtue of Theorem \ref{T:meassplitexist}.  
%On the other hand, Pedersen's up-down theorem  \cite[Theorem 2.4.3]{eilers2018c} implies that every self-adjoint element of $\pi(A)''$ is $\pi$-measurable. 
Hence $\psi$ and $\pi$ agree on the self-adjoint part of $A$  by Theorem \ref{T:meastrick}. It then easily follows that $\pi$ and $\psi$ agree on $A$.
\end{proof}

%%%%%%%%%%%%%%%%%%%%%%%%%%%%%%%%%%%%%%
%%%%%%%%%%%%%%%%%%%%%%%%%%%%%%%%%%%%%%
\subsection{A noncommutative Korovkin--\v Sa\v skin rigidity principle}\label{SS:ncKS}

Armed with an amended version of Arveson's hyperrigidity conjecture, we now explore  noncommutative analogues  of the Korovkin--\v Sa\v skin rigidity principle from the introduction. 

We introduce some terminology that will simplify the statements below.
Let $A$ and $B$ be unital $\rC^*$-algebras and let $\psi:A\to B$ be a unital completely positive map. Let $\Lambda$ be a directed set. For each $\lambda\in \Lambda$, assume that we are given an operator system $E_\lambda\subset A$ such that the net $(E_\lambda)$ is increasing (with respect to inclusion) and satisfies $A=\bigcup_{\lambda\in \Lambda}E_\lambda$. A net $\psi_\lambda:E_\lambda\to B$ of unital completely positive maps is a \emph{weak approximation} for $\psi$ on  some subset $X\subset A$ if the following property holds: for each $a\in X$ there is an index $\mu$ such that $a\in E_{\lambda}$ for every $\lambda\geq \mu$, and the net $(\psi_\lambda(a))_{\lambda\geq \mu}$ converges to $\psi(a)$ in the weak topology of $B$.

Let $\pi:A\to B(H)$ be a unital $*$-representation. Recall that we say that $\pi$ has the \emph{Korovkin rigidity property} with respect to $S$ if any net $\phi_\lambda:A\to \pi(A)$
of unital completely positive maps forming a weak approximation for $\pi$ on $S$ must necessarily be a weak approximation for $\pi$ on $A$. We emphasize here that the maps $\phi_\lambda$ must be defined on the entire $\rC^*$-algebra.

Our next goal is to connect the notions above to the unique tight extension property, for a large class of $\rC^*$-algebras. Recall that a unital $\rC^*$-algebra $A$ is \emph{locally reflexive} if for every finite-dimensional operator system $F\subset A^{**}$, there is a net of unital completely positive maps $\rho_i:F\rightarrow A$ such that, for each $x\in F$, the net $(\rho_i(x))$ converges to $x$ in the weak-$*$ topology of $A^{**}$.  
We note that exact (in particular, nuclear) $\rC^*$-algebras are locally reflexive \cite[Corollary 9.4.1]{brown2008textrm}.

We now establish the criterion on which our  main result relies.

\begin{theorem}\label{T:locref}
Let $A$ be a unital $\rC^*$-algebra and let $S\subset A$ be an operator system such that $\rC^*(S) = A$.  Let $\pi:A\to B(H)$ be a unital $*$-representation. Consider the following statements.
\begin{enumerate}[{\rm (i)}]
\item Let $\Upsilon_\pi:\pi(A)\to B(H_\pi)$ denote the universal representation of $\pi(A)$. Then, $\Upsilon_\pi\circ \pi$ has the unique tight extension property with respect to $S$.
\item Every weak approximation of $\pi:A\to \pi(A)$ on $S$ is necessarily a weak approximation on $A$.
\item $\pi:A\to \pi(A)$ has the Korovkin rigidity property with respect to $S$.
\end{enumerate}
Then, we have that ${\rm (i)}\Rightarrow {\rm (ii)}\Rightarrow {\rm (iii)}$. When $A$ is locally reflexive, we have ${\rm (i)}\Leftrightarrow {\rm (ii)}\Rightarrow {\rm (iii)}$. When $A$ is nuclear, we have ${\rm (i)}\Leftrightarrow {\rm (ii)}\Leftrightarrow {\rm (iii)}$.
\end{theorem}
\begin{proof}
By the well-known properties of the universal representation,  the weak-$*$ continuous surjective $*$-homomorphism $\widehat{\Upsilon_\pi}:\pi(A)^{**}\to (\Upsilon_\pi\circ\pi)(A)''$ is isometric. This will be used throughout the proof. 

(ii)$\Rightarrow$(iii): This is trivial.

(i)$\Rightarrow$(ii):  Let $\psi_\lambda:E_\lambda\to \pi(A)$ be a weak approximation of $\pi$ on $S$. For each $\lambda$, let $\Psi_\lambda:A\to B(H_\pi)$ be a unital completely positive extension of $\Upsilon_\pi\circ  \psi_\lambda$. Let $\Psi:A\to B(H_\pi)$ be a cluster point of $(\Psi_\lambda)$ in the pointwise weak-$*$ topology, so that $\Psi$ is also unital and completely positive. 

We now show that $\Psi$ and $\Upsilon_\pi\circ \pi$ agree on $S$.
Let $s\in S$ and choose $\mu_1$ such that $s\in E_\lambda$ for each $\lambda\geq \mu_1$. Then,  $(\psi_\lambda(s))_{\lambda\geq \mu_1}$ is a net in $\pi(A)$ converging to $\pi(s)$ in the weak topology of $\pi(A)$. Equivalently, $(\psi_\lambda(s))_{\lambda\geq \mu_1}$  converges to $\pi(s)$ in the weak-$*$ topology of $\pi(A)^{**}$. Applying $\widehat{\Upsilon_\pi}$, we find that, $((\Upsilon_\pi\circ \psi_\lambda)(s))_{\lambda\geq \mu_1}$ converges to $\Upsilon_\pi(\pi(s))$ in the weak-$*$ topology of $(\Upsilon_\pi\circ \pi)(A)''$. Therefore, $\Psi(s)=\Upsilon_\pi(\pi(s))$.

Next, we claim that $\Psi$ takes values in $(\Upsilon_\pi\circ \pi)(A)''$. To see this, fix $a\in A$ and choose $\mu_2$ such that $a\in E_\lambda$ for each $\lambda\geq \mu_2$. By construction,  $\Psi(a)$ is a weak-$*$ cluster point of the net $((\Upsilon_\pi\circ \psi_\lambda)(a))_{\lambda\geq \mu_2}$ in $B(H_\pi)$. Since this net lies in $(\Upsilon_\pi\circ \pi)(A)$, we can conclude that $\Psi(a)\in (\Upsilon_\pi\circ \pi)(A)''$, and the claim is established. Hence, $\Psi$ is a tight extension of $\Upsilon_\pi\circ \pi$, and our assumption then yields $\Psi=\Upsilon_\pi\circ\pi$.

We have  thus shown that any cluster point of the net $(\Psi_\lambda)$ in the poinwise weak-$*$ topology must coincide with $\Upsilon_\pi\circ \pi$, so by compactness we infer that the net $(\Psi_\lambda)$ itself converges to $\Upsilon_\pi\circ \pi$.

Finally, let $b\in A$ and choose $\mu_3$ such that $b\in E_\lambda$ for each $\lambda\geq \mu_3$. We know that the net $((\Upsilon_\pi\circ \psi_\lambda)(b))_{\lambda\geq \mu_3}$ lies in $(\Upsilon_\pi\circ \pi)(A)$ and converges to $(\Upsilon_\pi\circ\pi)(b)$ in the weak-$*$ topology of $(\Upsilon_\pi\circ \pi)(A)''$, by the previous paragraph. Applying $(\widehat{\Upsilon_\pi})^{-1}$, we infer in turn that the net $(\psi_\lambda(b))_{\lambda\geq \mu_3}$ converges to $\pi(b)$ in the weak-$*$ topology of $\pi(A)^{**}$. Since $(\psi_\lambda(b))_{\lambda\geq \mu_3}$ is a net in $\pi(A)$, this is equivalent to convergence in the weak topology of $\pi(A)$. Hence $\psi_\lambda:E_\lambda\to \pi(A)$ is indeed a weak approximation of $\pi$ on $A$.

(ii)$\Rightarrow$(i): We assume here that $A$ is locally reflexive. It follows from \cite[Proposition 5.3]{effros1985lifting} that $\pi(A)$ is also locally reflexive. Let $\Lambda$ denote the directed set of triples $(\eps,F,G)$ where $\eps>0$, $F\subset A$ is a finite-dimensional operator system and $G\subset \pi(A)^*$ is a finite subset. Let $\psi:A\to (\Upsilon_\pi\circ \pi)(A)''$ be a unital completely positive map agreeing with $\Upsilon_\pi\circ \pi$ on $S$. By local reflexivity, for each $\lambda=(\eps ,F,G) \in\Lambda$ there is a unital completely positive map $\phi_\lambda:(\widehat{\Upsilon_\pi}^{-1}\circ \psi)(F)  \to \pi(A)$ such that 
\begin{equation}\label{Eq:locref}
|\eta ( (\phi_\lambda\circ \widehat{\Upsilon_\pi}^{-1}\circ \psi)(a)-(\widehat{\Upsilon_\pi}^{-1}\circ \psi)(a))|<\eps  \quad  \text{ for every } a\in F, \eta\in G.
\end{equation}

For $s\in S$, there is $\mu_1$ such that if $\lambda=(\eps,F,G)\geq \mu_1$, then $s\in F$. It then follows from \eqref{Eq:locref}  that the net $( (\phi_\lambda\circ \widehat{\Upsilon_\pi}^{-1}\circ \psi)(s))_{\lambda\geq \mu_1}$ converges to $(\widehat{\Upsilon_\pi}^{-1}\circ \psi)(s)=\pi(s)$ in the weak-$*$ topology of $\pi(A)^{**}$. In turn, because $ (\phi_\lambda\circ \widehat{\Upsilon_\pi}^{-1}\circ \psi)(s)\in \pi(A)$ for each $\lambda$, we conclude that the net $( (\phi_\lambda\circ \widehat{\Upsilon_\pi}^{-1}\circ \psi)(s))_{\lambda\geq \mu_1}$  converges to $\pi(s)$ in the weak topology of $\pi(A)$. In other words, the net $ (\phi_\lambda\circ \widehat{\Upsilon_\pi}^{-1}\circ \psi) $ is a weak approximation for $\pi$ on $S$. By assumption, this implies that this net is a weak approximation for $\pi$ on all of $A$. 

Fix $a\in A$ and $\mu_2\in\Lambda$ such that $a\in F$ for each triple $(\eps,F,G)\geq \mu_2$. By \eqref{Eq:locref} again,  the net $( (\phi_\lambda\circ \widehat{\Upsilon_\pi}^{-1}\circ \psi)(a))_{\lambda\geq \mu_2}$ converges to $((\widehat{\Upsilon_\pi})^{-1}\circ \psi)(a)$ in the weak-$*$ topology of $\pi(A)^{**}$. On the other hand, the previous paragraph shows that it also converges to $\pi(a)$ in this topology, whence $(\Upsilon_\pi\circ \pi)(a)=\psi(a)$. We conclude that $\Upsilon_\pi\circ \pi$ has the unique tight extension property.

(iii) $\Rightarrow$ (i): We assume here that $A$ is nuclear. Let $\psi:A\to (\Upsilon_\pi\circ \pi)( A)''$ be a unital completely positive map agreeing with $\Upsilon_\pi\circ \pi$ on $S$. We let $\Lambda$ be the directed set of triples $(\eps,F,G)$ where $0<\eps<1$, and $F\subset A$ and $G\subset \pi(A)^*$ are finite subsets of norm one elements. 

For $\lambda=(\eps,F,G)\in \Lambda$, choose $0<\delta<\min\{\eps,\sqrt{3}/2\}$ small enough so that 
\begin{equation}\label{Eq:delta}
(1+\delta) \frac{1}{\sqrt{1-\delta}} \left(\frac{1}{\sqrt{1-\delta}}-1\right)<\eps.
\end{equation}
By nuclearity there is an integer $r\geq 1$ along with unital completely positive maps $\rho:A\to \bM_r$ and $\theta:\bM_r\to A$ such that
\begin{equation}\label{Eq:nuc}
\|(\theta\circ \rho)(a)-a\|<\delta, \quad a\in F.
\end{equation}
Consider  $\widehat {\Upsilon_\pi}^{-1}\circ \psi\circ\theta:\bM_r\to \pi(A)^{**}$. We now argue as in the proof of \cite[Proposition 2.3.8]{brown2008textrm} to find a completely positive map $\alpha_0:\bM_r\to  \pi(A)$ such that
\begin{equation}\label{Eq:kap}
|\eta((\alpha_0 \circ \rho)(a)-(\widehat {\Upsilon_\pi}^{-1}\circ \psi\circ \theta\circ \rho)(a))|<\delta  \quad \text{for every } a\in F, \eta\in G
\end{equation}
and $\|\alpha_0(1)-1\|<\delta$. The positive element $e= \alpha_0(1)\in A$ is then invertible with spectrum contained  in the interval $(1-\delta,1+\delta)$. Define  a unital completely positive map $\alpha:\bM_r\to A$ as 
\[
\alpha(t)=e^{-1/2} \alpha_0(t)e^{-1/2},\quad t\in \bM_r.
\]
Functional calculus reveals that $\|e^{-1/2}\|\leq 1/\sqrt{1-\delta}$ and 
\[\|e^{-1/2}-1\|\leq \max\left\{\frac{1}{\sqrt{1-\delta}}-1,1-\frac{1}{\sqrt{1+\delta}}\right\}=\frac{1}{\sqrt{1-\delta}}-1
\]
where we used that $\delta<\sqrt{3}/2$.
Combining these estimates with \eqref{Eq:delta},\eqref{Eq:nuc} and \eqref{Eq:kap}, a routine calculation yields
\begin{equation}\label{Eq:approx}
|\eta((\alpha \circ \rho)(a)-(\widehat {\Upsilon_\pi}^{-1} \circ \psi)(a))|<3\eps \quad \text{for every } a\in F, \eta\in G.
\end{equation}
Finally, put $\psi_\lambda=\alpha\circ \rho:A\to \pi(A)$. 

By \eqref{Eq:approx}, we see that the net $(\psi_\lambda)$ converges to $\widehat {\Upsilon_\pi}^{-1}\circ \psi$ pointwise in the weak-$*$ topology of $\pi(A)^{**}$.  
In particular, we see that $( \psi_\lambda(s))$ converges to $\pi(s)$ in the weak-$*$ topology of $\pi(A)^{**}$ for every $s\in S$. Therefore, the net $( \psi_\lambda)$ is a weak approximation of $\pi$ on $S$. By the Korovkin rigidity property, it must be a weak approximation for $\pi$ on $A$, and we infer that $( \psi_\lambda(a))$ converges to $\pi(a)$ in the weak-$*$ topology of $\pi(A)^{**}$ for each $a\in A$. Applying $\widehat{\Upsilon_\pi}$, we find that $((\Upsilon_\pi\circ \psi_\lambda)(a))$ converges to $(\Upsilon_\pi\circ \pi)(a)$ in the weak-$*$ topology of $(\Upsilon_\pi\circ\pi)(A)''$, whence $\psi(a)=(\Upsilon_\pi\circ \pi)(a)$. We conclude that $\Upsilon_\pi\circ \pi$ has the unique tight extension property, and the proof is complete.
\end{proof}

Next,  we record a standard fact.

\begin{lemma}\label{L:reptrick}
Let $A$ be a unital $\rC^*$-algebra and let $S\subset A$ be an operator system such that $A=\rC^*(S)$. Suppose $\pi$ and $\sigma$ are $*$-representations of $A$ and that there is a weak-$*$ continuous surjective $*$-homomorphism $\Omega:\pi(A)''\to \sigma(A)''$ such that $\Omega\circ \pi=\sigma$. Then, $\sigma$ has the unique tight extension property with respect to $S$ whenever $\pi$ has this property.
\end{lemma}

\begin{proof}
Let $z\in \pi(A)''$ be the central projection such that $\ker \Omega=\pi(A)''(1-z)$. The restriction $\Omega_0=\Omega|_{z \pi(A)''}: z\pi(A)''\to \sigma(A)''$ is then a $*$-isomorphism. If $\psi:A\to \sigma(A)''$ is a unital completely positive map such that $\psi|_S=\sigma|_S$, then one may define a unital completely positive map $\phi:A\to \pi(A)''$ by
\[
\phi(a)=(\Omega_0^{-1}\circ \psi)(a)+ \pi(a)(1-z), \quad a\in A.
\]
A standard verification yields $\varphi|_S = \pi|_S$. Since $\pi$ is assumed to have the unique tight extension property with respect to $S$, we find $\phi=\pi$. In particular, for each $a\in A$ we find
\[
\Omega_0^{-1}(\psi(a))=z\phi(a)=z\pi(a)=\Omega_0^{-1}(\sigma(a))
\]
and hence $\psi(a)=\sigma(a)$ as desired.
\end{proof}

The following is our noncommutative analogue of the Korovkin--\v Sa\v skin rigidity principle, and is our second main result.

\begin{corollary}\label{C:nuclearrigprinc}
Let $A$ be a unital, separable, nuclear $\rC^*$-algebra and let $S\subset A$ be an operator system such that $A=\rC^*(S)$. Then, the following statements are equivalent.
\begin{enumerate}[{\rm (i)}]
\item Every irreducible $*$-representation of $A$ is a boundary representation for $S$.
\item The identity representation of $A$ has the  Korovkin rigidity property with respect to $S$.
%\item Every unital $*$-representation of $A$ has the unique tame extension property with respect to $S$.
\item Every unital $*$-representation $\pi:A\to \pi(A)$ has the Korovkin rigidity property with respect to $S$.
\end{enumerate}
\end{corollary}

\begin{proof}
(iii) $\Rightarrow$ (ii): This is trivial.

(ii) $\Rightarrow$ (i): By Theorem \ref{T:locref}, we see that the universal representation of $A$ has the unique tight extension property with respect to $S$. In turn, by Lemma \ref{L:reptrick} we see that every unital $*$-representation of $A$ has the unique tight extension property with respect to $S$. Applying Theorem \ref{T:tightHR} we see that (i) holds.

(i) $\Rightarrow$ (iii): By Theorem \ref{T:tightHR}, we infer that every unital $*$-representation of $A$ has the unique tight extension property with respect to $S$. Hence, Theorem \ref{T:locref} implies that every unital $*$-representation $\pi:A\to \pi(A)$ has the Korovkin rigidity property with respect to $S$.
\end{proof}

A similar rigidity principle can be stated in the locally reflexive case, the details of which we leave to the interested reader.

%%%%%%%%%%%%%%%%%%%%%%%%%%%%%%%%%%%%%%
%%%%%%%%%%%%%%%%%%%%%%%%%%%%%%%%%%%%%%

\section{The strong Korovkin rigidity property}\label{S:NormK}

Upon comparing the classical Korovkin--\v Sa\v skin rigidity principle with our noncommutative counterpart (Corollary \ref{C:nuclearrigprinc}), one immediately notices a discrepancy. Indeed, the classical version  is concerned with approximations in the norm topology,  while in the noncommutative world we are dealing with the weak topology. In this section, we address this topological issue.

Let $A$ be a unital $\rC^*$-algebra and let $S\subset A$ be an operator system such that $\rC^*(S) = A$. We say that a unital $*$-representation $\pi: A\rightarrow B(H)$ has the \emph{strong Korovkin rigidity  property} with respect to $S$ if, whenever $\phi_\lambda:A\rightarrow\pi(A)$ is a net of unital completely positive maps satisfying
\[
\lim_\lambda \|\phi_\lambda(s)-\pi(s)\|=0 \quad \ \text{ for every } s\in S 
\]
it follows that
\[
\lim_\lambda \|\phi_\lambda(a)-\pi(a)\|=0 \quad \ \text{ for every } a\in A.
\]
We aim to find sufficient conditions for the strong Korovkin rigidity property to hold. We approach this question by investigating the dependence of this property on the representation. 

For instance, an elementary argument shows that if a $*$-representation $\pi$ of $A$ has the strong Korovkin rigidity property with respect to $S$, and if $\sigma$ is another $*$-representation with $\ker \sigma=\ker \pi$, then $\sigma$ also has this rigidity property.

Our next goal is to show that a more powerful  statement holds for certain well-behaved $\rC^*$-algebras. Recall that a unital $\rC^*$-algebra $A$ has the \emph{lifting property} if, given a surjective $*$-homomorphism $\Omega: B\rightarrow D$ between unital $\rC^*$-algebras and a unital completely positive map $\psi: A\rightarrow D$, there exists a unital completely positive map $\Psi:A\rightarrow B$ such that $\Omega\circ\Psi = \psi$. The class of $\rC^*$-algebras with the lifting property includes separable nuclear $\rC^*$-algebras \cite{choi1976completely} (see also \cite{pisier2020non} for a recent account on this topic).

We require the following technical fact, essentially due to Arveson.

\begin{lemma}\label{L:lift}
Let $A$ be a unital $\rC^*$-algebra with the lifting property. Let $\pi$ and $\sigma$ be two unital $*$-representations such that there is a unital surjective $*$-homomorphism $\Omega:\pi(A)\to \sigma(A)$ with $\Omega\circ \pi=\sigma$. Let $\psi:A\to \sigma(A)$ be a unital completely positive map, let $F\subset A$ be a finite set and let $\eps>0$. Then, there is a unital completely positive map $\Psi:A\to \pi(A)$ such that $\Omega\circ \Psi=\psi$ and
\[
\|\Psi(a)-\pi(a)\|\leq \| \psi(a)-\sigma(a)\|+\eps, \quad a\in F.
\]
\end{lemma}
\begin{proof}
This follows from the argument given in \cite[Lemma 3.1]{arveson1977notes}.
\end{proof}

We can now prove a change of representation result for the strong Korovkin rigidity property.

\begin{theorem}\label{T:lift}
Let $A$ be a unital $\rC^*$-algebra with the lifting property. Let $\pi$ and $\sigma$ be two unital $*$-representations such that $\ker\pi\subset \ker \sigma$. If $\pi$ has the strong Korovkin rigidity property with respect to $S$, then so does $\sigma$.
\end{theorem}
\begin{proof}

Let $\Lambda$ be a directed set, and let $\psi_\lambda:A\to \sigma(A)$ be a net of unital completely positive maps such that
\[
\lim_{\lambda}\|\psi_\lambda(s)-\sigma(s)\|=0, \quad s\in S.
\]
Let $\Gamma$ be the set of triples $(\lambda, F,\eps)$ where $\lambda\in \Lambda$, $F\subset A$ is finite and $\eps>0$. 
By assumption, there is a unital surjective $*$-homomorphism $\Omega:\pi(A)\to \sigma(A)$ such that $\Omega\circ\pi=\sigma$. Thus, for each $\gamma=(\lambda,F,\eps)\in \Gamma$, by virtue of Lemma \ref{L:lift} there is a unital completely positive map $\Psi_\gamma:A\to \pi(A)$ such that 
$\Omega\circ \Psi_\gamma =\psi_\lambda$ and
\[
\|\Psi_\gamma (a)-\pi(a)\|\leq \| \psi_\lambda (a)-\sigma(a)\|+\eps, \quad a\in F.
\]
It is easily seen then that 
\[
\lim_\gamma \|\Psi_\gamma(s)-\pi(s)\|=0, \quad s\in S.
\]
Since $\pi$ is assumed to have the strong Korovkin rigidity property, we infer that
\[
\lim_\gamma \|\Psi_\gamma(a)-\pi(a)\|=0, \quad a\in A.
\]
Applying $\Omega$, we find
\[
\lim_\gamma \|(\Omega\circ\Psi_\gamma)(a)-\sigma(a)\|=0, \quad a\in A.
\]
Finally, fix $a\in A$ and $\eps>0$. There is $\gamma_0=(\lambda_0,F_0,\eps_0)\in \Gamma$ such that 
\[
 \|(\Omega\circ\Psi_\gamma)(a)-\sigma(a)\|<\eps
\]
if $\gamma\geq \gamma_0$. Given  $\lambda\geq \lambda_0$, we put $\gamma=(\lambda,F_0,\eps_0)\in \Gamma$. Then, $\gamma\geq \gamma_0$ whence
\[
\|\psi_{\lambda}(a)-\sigma(a)\|=\|(\Omega\circ\Psi_\gamma)(a)-\sigma(a)\|<\eps.
\]
This shows that 
\[
\lim_{\lambda}\|\psi_\lambda(a)-\sigma(a)\|=0, \quad a\in A
\]
as desired.
\end{proof}

If the lifting property is replaced by the (formally weaker) local lifting property, then the previous argument can be adapted and a similar statement to Theorem \ref{T:lift} can be given, where the strong Korovkin rigidity property is replaced by the obvious strong variant of the notion of approximation studied in Subsection \ref{SS:ncKS}. 
We leave the details to the interested reader.   We also mention here that it is not known whether these two lifting properties are equivalent, and in fact this problem is at the heart of the Connes--Kirchberg conjecture \cite{pisier2023lifting}.

Next, we show how to connect the strong Korovkin rigidity property to the unique tight extension property, for a subclass of the nuclear $\rC^*$-algebras.

Recall that a $\rC^*$-algebra is said to be \emph{homogeneous} when all its irreducible $*$-representations have the same finite dimension. For this class, we can connect the strong Korovkin rigidity property to the unique tight extension property, by proving the following variation on Theorem \ref{T:locref}.

\begin{theorem}\label{T:normhomo}
Let $A$ be a unital  $\rC^*$-algebra. Let $\pi:A\to B(H)$ be a unital $*$-representation and let $\Upsilon_\pi$ denote the universal representation of $\pi(A)$. If $\pi(A)$ is homogeneous and $\Upsilon_\pi\circ \pi$ has the unique tight extension property, then $\pi$ has the strong Korovkin rigidity property with respect to $S$.
\end{theorem}
\begin{proof}
Since $\pi(A)$ is homogeneous, there is a positive integer $d$ such that every irreducible $*$-representation of $\pi(A)$ is $d$-dimensional. Fix once and for all a Hilbert space $F$ of dimension $d$. Then, every irreducible $*$-representation of $\pi(A)$ is unitarily equivalent to an irreducible $*$-representation on $F$. 

Let $\psi_\gamma:A\to \pi(A)$ be a net of unital completely positive maps such that 
\[
\lim_\gamma\|\psi_\gamma(s)-\pi(s)\|=0, \quad s\in S.
\]
Fix a self-adjoint element $a\in A$. The goal is to show that the net $(\psi_\gamma(a))$ converges to $\pi(a)$ in norm. Equivalently, we must show that any subnet of $(\psi_\gamma(a))$ has a further subnet converging in norm to $\pi(a)$.

Fix  a subnet  $(\psi_\lambda(a))$ of $(\psi_\gamma(a))$. For each $\lambda$, by the first paragraph of the proof we may find an irreducible $*$-representation $\sigma_\lambda:\pi(A)\to B(F)$ such that 
\begin{equation}\label{Eq:normrep}
\|\sigma_\lambda(\psi_\lambda(a)-\pi(a))\|=\|\psi_\lambda(a)-\pi(a)\|.
\end{equation}
Let $\sigma:\pi(A)\to B(F)$ be a cluster point of the net $(\sigma_\lambda)$ in the pointwise norm topology, so that there is a subnet $(\sigma_{\lambda_\mu})$ satisfying
\begin{equation}\label{Eq:conv1}
\lim_\mu\|\sigma(\pi(b))-\sigma_{\lambda_\mu}(\pi(b))\|=0, \quad b\in A.
\end{equation}
Then, $\sigma$ is a unital $*$-representation of $\pi(A)$, and it must necessarily be irreducible since $\pi(A)$ is homogeneous.

We claim that 
\begin{equation}\label{Eq:conv2}
\lim_\mu \|\sigma_{\lambda_\mu}(\psi_{\lambda_\mu}(b))-\sigma(\pi(b))\|=0, \quad b\in A.
\end{equation}
To see this, consider the weak-$*$ continuous $*$-representation $\widehat\sigma:\pi(A)^{**}\to B(F)$ and let $z\in \pi(A)^{**}$ be the central projection such that $\ker \widehat\sigma=\pi(A)^{**}(1-z)$. Since $F$ is finite-dimensional and $\sigma$ is irreducible, we have $B(F)=\sigma(\pi(A))$ and there is a weak-$*$ homeomorphic $*$-isomorphism $\Omega:B(F)\to \pi(A)^{**}z$ such that $\Omega(\sigma(\pi(a)))=\pi(a)z$ for each $a\in A$.  For each $\mu$, we define a unital completely positive map $\Psi_\mu:A\to \pi(A)^{**}$ as
\[
\Psi_\mu(b)=(\Omega\circ \sigma_{\lambda_\mu} \circ \psi_{\lambda_\mu})(b)\oplus \pi(b)(1-z)
\]
for every $b\in A$. By construction, for each $s\in S$ we see that
\[
\lim_\mu \|\sigma_{\lambda_\mu}(\psi_{\lambda_\mu}(s))-\sigma(\pi(s))\|=0
\]
and so
\[
\lim_\mu \|\Psi_\mu(s)-\pi(s)\|=0.
\]
Let $\Phi:A\to \Upsilon_\pi(A)''$ be a cluster point of the net $(\Upsilon_\pi\circ \Psi_\mu)$ in the pointwise weak-$*$ topology. Then, $\Phi$ is a unital completely positive map agreeing with $\Upsilon_\pi\circ \pi$ on $S$. By the unique tight extension property of $\Upsilon_\pi\circ \pi$, we infer that $\Phi=\Upsilon_\pi\circ \pi$ on $A$. Applying $\widehat{\Upsilon_\pi}^{-1}$, we find that the net $(\Psi_\mu(b))$ converges to $\pi(b)$ in the weak-$*$ topology of $\pi(A)^{**}$ for every $b\in A$. In particular,  for $b\in A$ the net $((\Omega\circ \sigma_{\lambda_\mu} \circ \psi_{\lambda_\mu})(b))$ converges to $\pi(b)z$  in the weak-$*$ topology of $\pi(A)^{**}$ and, in turn, by applying $\Omega^{-1}$ we see that $((\sigma_{\lambda_\mu} \circ \psi_{\lambda_\mu})(b))$ converges to $\sigma(\pi(b))$  in the norm topology of $B(F)$ (since $F$ is finite-dimensional). This establishes the claim.

Finally, upon combining \eqref{Eq:conv1} and \eqref{Eq:conv2}, we find 
\[
\lim_\mu \|\sigma_{\lambda_\mu}(\pi(a)-\psi_{\lambda_\mu}(a))\|=0.
\]
By \eqref{Eq:normrep}, this implies 
\[
\lim_\mu \|\pi(a)-\psi_{\lambda_\mu}(a)\|=0
\]
as desired.

\end{proof}

As a simple application, we give a criterion that guarantees the strong Korovkin rigidity property for representations with homogeneous ranges. This will be of use later.

\begin{corollary}\label{C:homogrange}
Let $A$ be a separable unital  $\rC^*$-algebra and let $S\subset A$ be an operator system such that $\rC^*(S)=A$. Let $\pi:A\to B(H)$ be a unital $*$-representation such that $\pi(A)$ is homogeneous. If all irreducible $*$-representations of $A$ are boundary representations for $S$, then $\pi$ has the strong Korovkin rigidity property with respect to $S$.
\end{corollary}
\begin{proof}
By virtue of Theorem \ref{T:normhomo}, it suffices to check that $\Upsilon_\pi\circ \pi$ has the unique tight extension property. This follows from Theorem \ref{T:tightHR}.
\end{proof}

We can now give a strong version of the noncommutative Korovkin--\v Sa\v skin rigidity principle, where we replace the commutativity condition on the algebra by homogeneity.

\begin{corollary}\label{C:KShomo}
Let $A$ be a separable, homogeneous, unital $\rC^*$-algebra and let $S\subset A$ be an operator system such that $\rC^*(S)=A$. Then, the following statements are equivalent.
\begin{enumerate}[{\rm (i)}]
\item Every irreducible $*$-representation of $A$ is a boundary representation for $S$.
\item The identity representation of $A$ has the strong Korovkin rigidity property with respect to $S$.
\item Every unital $*$-representation $\pi:A\to\pi(A)$ has the strong Korovkin rigidity property with respect to $S$.
\end{enumerate}
\end{corollary}
\begin{proof}
{\rm (i)}$\Rightarrow${\rm (iii)}: Since $A$ is homogeneous, so is $\pi(A)$ for every $*$-representation $\pi$. Hence, the desired conclusion follows from Corollary \ref{C:homogrange}.

{\rm (iii)}$\Rightarrow${\rm (i)}: Let $\pi:A\to B(H)$ be an irreducible $*$-representation and let $\psi:A\to B(H)$ be a unital completely positive map agreeing with $\pi$ on $S$. Because $A$ is homogeneous, $H$ must be finite-dimensional whence  $\pi(A)=B(H)$. The strong Korovkin rigidity property applied to the constant net $\psi$ yields that $\pi=\psi$ on $A$. Hence, $\pi$ is a boundary representation for $S$.

{\rm (iii)}$\Rightarrow${\rm (ii)}: Obvious.

{\rm (ii)}$\Rightarrow${\rm (iii)}: This follows immediately from Theorem \ref{T:lift}, seeing as homogeneous $\rC^*$-algebras are nuclear \cite[Proposition 2.7.7]{brown2008textrm} and hence have the lifting property.
\end{proof}

Combining the previous result with Corollary \ref{C:nuclearrigprinc}, we see that the Korovkin rigidity propertiy is equivalent to its strong variant for the identity representation of a separable homogeneous $\rC^*$-algebras. We do not know whether this equivalence is valid for general separable nuclear $\rC^*$-algebras.

%%%%%%%%%%%%%%%%%%%%%%%%%%%%%%%%%%%%%%

%%%%%%%%%%%%%%%%%%%%%%%%%%%%%%%%%%%%%%

\section{Rigidity of the Bilich--Dor-On counterexample}

We close the paper by analyzing the recent counterexample to Arveson's hyperrigidity conjecture by Bilich and Dor-On through the lens of our main results. 

Let $B\subset B(H)$ be the norm-closed unital subalgebra constructed in \cite{bilich2024arveson}, and let $S\subset B(H)$ be the operator system generated by $B$. For our purposes, the relevant features of the construction are the following.
\begin{enumerate}
\item The Hilbert space $H$ is separable.
\item The $\rC^*$-algebra $\rC^*(S)$ contains the ideal $\fK$  of compact operators on $H$, and $\rC^*(S)/\fK$ is commutative. In particular, $\rC^*(S)$ is nuclear.
\item All irreducible $*$-representations of $\rC^*(S)$ are boundary representations for $S$.
\item There is a separable Hilbert space $H'$ and a unital $*$-representation $\pi:\rC^*(S)\to B(H')$ that does not have the unique extension property with respect to $S$ and such that $\fK\subset \ker \pi$.
\end{enumerate}
We shall refer to $S$ and $\pi$ above as the \emph{Bilich--Dor-On operator system} and \emph{Bilich--Dor-On representation}, respectively.

The operator system $S$ does not display the kind of rigidity behaviour that Arveson's conjecture predicted. One is then left, as pointed out by the authors of \cite{bilich2024arveson}, to wonder if $S$ enjoys some other type of rigidity. The following provides an answer to this question.

\begin{theorem}\label{T:BD}
Let $S$ be the Bilich--Dor-On operator system and let $\pi$ be the Bilich--Dor-On representation. Then, the following statements hold.
\begin{enumerate}[\rm (i)]
\item Every unital $*$-representation of $\rC^*(S)$ has the unique tight extension property with respect to $S$.
\item Every unital $*$-representation of $\rC^*(S)$ has the  Korovkin rigidity property with respect to $S$.
\item The representation $\pi$ has the strong Korovkin rigidity property.
\end{enumerate}
\end{theorem}
\begin{proof}
Statement (i) follows directly from Theorem \ref{T:tightHR}. Next, note that $\rC^*(S)$ is nuclear. Hence, statement (ii) follows from Corollary \ref{C:nuclearrigprinc}. Finally, since $\pi$ annihilates $\fK$ and $\rC^*(S)/\fK$ is commutative, it follows that $\pi(\rC^*(S))$ is commutative, so statement (iii) is a consequence of Corollary \ref{C:homogrange}.
\end{proof}

\bibliographystyle{plain}
\bibliography{korovkin}
	
\end{document}